\documentclass[12pt]{article}
\usepackage{amscd}
\usepackage{amsfonts}
\usepackage{CJK}
\usepackage{amsthm}
\usepackage{amsmath}
\usepackage{indentfirst}
\usepackage[top=1em,bottom=1em,left=6em,right=6em,includehead,includefoot]{geometry}
\begin{document}
\begin{CJK}{GBK}{song}

\newtheorem{theorem}{Theorem}
\newtheorem{lemma}{Lemma}
\newtheorem{definition}{Definition}
\newtheorem{remark}{Remark}
\newtheorem{corollery}{Corollery}

\title{\bf The Cartan Model for Equivariant Cohomology}
\author{Xu Chen \footnote{{\it Email:} xiaorenwu08@163.com.  ChongQing, China }}
\date{}
\maketitle

\begin{abstract}
In this article, we will discuss a new operator $d_{C}$ on $W(\mathfrak{g})\otimes\Omega^{*}(M)$ and to construct a new Cartan model for equivariant cohomology. We use the new Cartan model to construct the corresponding BRST model and Weil model, and discuss the relations between them.
\end{abstract}

\section{Introduction}
The standard Cartan model for equivariant cohomology is construct on the algebra $W(\mathfrak{g})\otimes\Omega^{*}(M)$ with operator $$d_{C}\phi^{i}=0, \phi^{i}\in S(\mathfrak{g}^{*}), i=1,\cdots,n;$$
$$d_{C}\eta=(1\otimes d-\sum_{b=1}^{n}\phi^{b}\otimes\iota_{b})\eta, \eta\in\Omega^{*}(M),$$
where $\iota_{b}$ is $\iota_{e_{b}}$(see [4],[5],[7],[8]). We can also introduce a new operator on $W(\mathfrak{g})\otimes\Omega^{*}(M)$ by $$d_{C}\phi^{i}=0, \phi^{i}\in S(\mathfrak{g}^{*}), i=1,\cdots,n;$$
$$d_{C}\eta=(1\otimes d-\sum_{b=1}^{n}\phi^{b}\otimes(\iota_{b}+\sqrt{-1}f_{b}^{a}\iota_{a}))\eta, \eta\in\Omega^{*}(M)\otimes\mathbb{C},$$
where $\iota_{b}$ is $\iota_{e_{b}}$. In this article we construct the new model for equivariant cohomology which also called Cartan model. The idea comes form the article [3]. We also use the new Cartan model to construct the corresponding BRST model and Weil model.

\section{Cartan model}
Let $G$ ba a compact Lie group with Lie algebra $\mathfrak{g}$,
$\mathfrak{g}^{*}$ be the dual of $\mathfrak{g}$. We known the Weil
algebra is
$$W(\mathfrak{g})=\wedge(\mathfrak{g}^{*})\otimes S(\mathfrak{g}^{*}).$$
The contraction $i_{X}$ and the exterior derivative $d_{W}$ on
$W(\mathfrak{g})$ defined as follow: \par

Choose a basis $e_{1},\cdots,e_{n}$ for $\mathfrak{g}$ and let
$e_{1}^{*},\cdots,e_{n}^{*}$ be the dual basis of
$\mathfrak{g}^{*}$. Let $\theta^{1},\cdots,\theta^{n}$ be the dual
basis of $\mathfrak{g}^{*}$ generating the exterior algebra
$\wedge(\mathfrak{g}^{*})$ and let $\phi^{1},\cdots,\phi^{n}$ be the
dual basis of $\mathfrak{g}^{*}$ generating the symmetric algebra
$S(\mathfrak{g}^{*})$. Let $c^{i}_{jk}$ be the structure constants
of $\mathfrak{g}$(see [6]), that is
$[e_{j},e_{k}]=\Sigma_{i=1}^{n}c^{i}_{jk}e_{i}$.
We kown that $S(\mathfrak{g}^{*})$ is identified with the polynomial ring $\mathbb{C}[\phi^{1},\cdots,\phi^{n}]$.\par

Define the contraction $i_{X}$ on $W(\mathfrak{g})$ for any
$X\in\mathfrak{g}$ by
$$i_{e_{r}}(\theta^{s})=\delta_{r}^{s}, \ i_{e_{r}}(\phi^{s})=0$$
for all $r,s=1,\cdots,n$ and extending by linearity and as a
derivation.

Define $d_{W}$ by
$$d_{W}\theta^{i}=-\frac{1}{2}\sum_{j,k}c^{i}_{jk}\theta^{j}\wedge\theta^{k}+\phi^{i}$$
and $$d_{W}\phi^{i}=-\sum_{j,k}c^{i}_{jk}\theta^{j}\phi^{k}$$ and
extending $d_{W}$ to $W(\mathfrak{g})$ as a derivation.

The Lie derivative on $W(\mathfrak{g})$ is defined by
$$L_{X}=d_{W}\cdot i_{X}+i_{X}\cdot d_{W}.$$

\begin{lemma}
$L_{e_{i}}\theta^{j}=-\sum_{k}c^{j}_{ik}\theta^{k}$ and
$L_{e_{i}}\phi^{j}=-\sum_{k}c^{j}_{ik}\phi^{k}$.
\end{lemma}
\begin{proof}
Because
$$L_{e_{i}}\theta^{j}=(d_{W}\cdot i_{e_{i}}+i_{e_{i}}\cdot d_{W})\theta^{j}=i_{e_{i}}(-\frac{1}{2}\sum_{i,k}c^{j}_{ik}\theta^{i}\wedge\theta^{k}+\phi^{j})=-\sum_{k}c^{j}_{ik}\theta^{k},$$
$$L_{e_{i}}\phi^{j}=(d_{W}\cdot i_{e_{i}}+i_{e_{i}}\cdot d_{W})\phi^{j}=i_{e_{i}}(-\sum_{i,k}c^{j}_{ik}\theta^{i}\phi^{k})=-\sum_{k}c^{j}_{ik}\phi^{k}$$
\end{proof}

\begin{lemma}
The operators $i_{X}, d_{W}, L_{X}$ on $W(\mathfrak{g})$ satisfy the
following identities:
\begin{description}
\item[(1)]$d_{W}^{2}=0$;
\item[(2)]$L_{X}\cdot d_{W}-d_{W}\cdot L_{X}=0$, for any
$X\in\mathfrak{g}$;
\item[(3)]$i_{X}i_{Y}+i_{Y}i_{X}=0$, for any $X,Y\in\mathfrak{g}$;
\item[(4)]$L_{X}i_{Y}-i_{Y}L_{X}=i_{[X,Y]}$, for any
$X,Y\in\mathfrak{g}$;
\item[(5)]$L_{X}L_{Y}-L_{Y}L_{X}=L_{[X,Y]}$, for any
$X,Y\in\mathfrak{g}$;
\item[(6)]$d_{W} i_{X}+i_{X} d_{W}=L_{X}$, for any
$X\in\mathfrak{g}$.
\end{description}
\end{lemma}
\begin{proof}
see [4].
\end{proof}
So, there is a complex $(W(\mathfrak{g}), d_{W})$, the cohomology of
$(W(\mathfrak{g}), d_{W})$ is trivial (see [5]), i.e. $H^{*}(W(\mathfrak{g}))\cong \mathbb{R}$.

Let $M$ be a smooth closed manifold with $G$ acting smoothly on the
left. Let $X^{M}$ be the vector field generated by the Lie algebra
element $X\in\mathfrak{g}$ given by
$$(X^{M}f)(x)=\frac{d}{dt}f(\exp(-tX)\cdot x)\mid_{t=0}.$$
Set $d,\iota_{X^{M}},\mathcal{L}_{X^{M}}$ be the exterior
derivative, contraction and Lie derivative on $\Omega^{*}(M)$.
Denote $\iota_{X}=\iota_{X^{M}}$ and $\mathcal
{L}_{X}=\mathcal{L}_{X^{M}}$ acting on $\Omega^{*}(M)$.

\begin{definition}
The Cartan model is defined by the algebra $$S(\mathfrak{g}^{*})\otimes\Omega^{*}(M)$$
and the differential $$d_{C}\phi^{i}=0, \phi^{i}\in S(\mathfrak{g}^{*}), i=1,\cdots,n;$$
$$d_{C}\eta=(1\otimes d-\sum_{i=1}^{n}\phi^{i}\otimes(\iota_{i}+\sqrt{-1}f_{i}^{j}\iota_{j}))\eta, \eta\in\Omega^{*}(M)\otimes\mathbb{C},$$
where $\iota_{i}$ is $\iota_{e_{i}}$ and $f_{i}^{j}\in\mathbb{R}$.
The operator $d_{C}$ is called the equivariant exterior derivative.
\end{definition}

Its action on forms $\alpha\in S(\mathfrak{g}^{*})\otimes\Omega^{*}(M)$ is
$$(d_{C}\alpha)(X)=(d-\iota_{X^{M}}-\sqrt{-1}\iota_{Y^{M}})(\alpha(X))$$
where $X^{M}=c^{i}X_{i}^{M}$ is the vector field on M generated by the Lie algebra element $X=c^{i}e_{i}\in\mathfrak{g}$,$Y^{M}=f_{j}^{i}c^{j}X_{i}^{M}$(see [2]). In the artile [3] we use the operator $d+\iota_{X^{M}}+\sqrt{-1}\iota_{Y^{M}}$ to construct an complex $(\Omega^{*}(M)\otimes\mathbb{C}, d+\iota_{X^{M}}+\sqrt{-1}\iota_{Y^{M}})$ and cohomology group $H^{*}_{X+\sqrt{-1}Y}(M)$, we can do it in the same way by the operator $d-\iota_{X^{M}}-\sqrt{-1}\iota_{Y^{M}}$.

\begin{lemma}
$$d^{2}_{C}=-\sum_{i=1}^{n}\phi^{i}\otimes(\mathcal{L}_{i}+\sqrt{-1}f_{i}^{j}\mathcal{L}_{j})$$
\end{lemma}
\begin{proof}
By the lemma 2. we have
\begin{align*}
d^{2}_{C} & =(1\otimes d-\sum_{i=1}^{n}\phi^{i}\otimes(\iota_{i}+\sqrt{-1}f_{i}^{j}\iota_{j}))(1\otimes d-\sum_{i=1}^{n}\phi^{i}\otimes(\iota_{i}+\sqrt{-1}f_{i}^{j}\iota_{j}))\\
          & = -\sum_{i=1}^{n}\phi^{i}\otimes [d(\iota_{i}+\sqrt{-1}f_{i}^{j}\iota_{j}))+(\iota_{i}+\sqrt{-1}f_{i}^{j}\iota_{j}))d]\\
          & = -\sum_{i=1}^{n}\phi^{i}\otimes(\mathcal{L}_{i}+\sqrt{-1}f_{i}^{j}\mathcal{L}_{j})
              \end{align*}
\end{proof}

Let $(S(\mathfrak{g}^{*})\otimes\Omega^{*}(M))^{\widetilde{G}}$ be the subalgebra of $S(\mathfrak{g}^{*})\otimes\Omega^{*}(M)$ which satisfied
$$(\sum_{i=1}^{n}\phi^{i}\otimes(\mathcal{L}_{i}+\sqrt{-1}f_{i}^{j}\mathcal{L}_{j}))\alpha=0, \forall\alpha\in(S(\mathfrak{g}^{*})\otimes\Omega^{*}(M))^{\widetilde{G}}$$
So we get the complex $((S(\mathfrak{g}^{*})\otimes\Omega^{*}(M))^{\widetilde{G}},d_{C})$.
The equivariantly closed form is $\forall\alpha\in(S(\mathfrak{g}^{*})\otimes\Omega^{*}(M))^{\widetilde{G}}$ with $d_{C}\alpha=0$,
the equivariantly exact form is $\forall\alpha\in(S(\mathfrak{g}^{*})\otimes\Omega^{*}(M))^{\widetilde{G}}$ there is $\beta\in(S(\mathfrak{g}^{*})\otimes\Omega^{*}(M))^{\widetilde{G}}$ with $\alpha=d_{C}\beta$.

As in [8] we can define the equivariant connection
$$\nabla_{\mathfrak{g}}=1\otimes\nabla-\sum_{i=1}^{n}\phi^{i}\otimes(\iota_{i}+\sqrt{-1}f_{i}^{j}\iota_{j})$$
and the equivariant curvature of the connection
$$F_{\mathfrak{g}}=(\nabla_{\mathfrak{g}})^{2}+\sum_{i=1}^{n}\phi^{i}\otimes(\mathcal{L}_{i}+\sqrt{-1}f_{i}^{j}\mathcal{L}_{j})$$

\section{BRST model}
This section is inspired by [5].
First, we will to construct the BRST differential algebra. The
algebra is
$$B=W(\mathfrak{g})\otimes\Omega^{*}(M).$$
The BRST operator is
$$\delta=d_{W}\otimes 1+1\otimes d+\sum_{i=1}^{n}\theta^{i}\otimes(\mathcal{L}_{i}+\sqrt{-1}f_{i}^{j}\mathcal{L}_{j})-\sum_{a=1}^{n}\phi^{a}\otimes(\iota_{a}+\sqrt{-1}f_{a}^{b}\iota_{b})+\frac{1}{2}\sum_{j,k}c_{jk}^{i}\theta^{j}\theta^{k}\otimes(\iota_{i}+\sqrt{-1}f_{i}^{j}\iota_{j})$$
$$-\sum_{j<k}\theta^{j}\theta^{k}\otimes((\mathcal{L}_{j}+\sqrt{-1}f_{j}^{h}\mathcal{L}_{h})(\iota_{k}+\sqrt{-1}f_{k}^{g}\iota_{g})-(\iota_{j}+\sqrt{-1}f_{j}^{h}\iota_{h})(\mathcal{L}_{k}+\sqrt{-1}f_{k}^{g}\mathcal{L}_{g}))$$
where $\mathcal{L}_{i}$ is $\mathcal{L}_{e_{i}}$ and $\iota_{a}$ is
$\iota_{e_{a}}$.

\begin{lemma}
On the algebra
$W(\mathfrak{g})\otimes\Omega^{*}(M)$, we have
$\delta^{2}=0$.
\end{lemma}
\begin{proof}
By computation, we have $$\delta=\exp(\sum_{i=1}^{n}\theta^{i}\otimes(\iota_{i}+\sqrt{-1}f_{i}^{j}\iota_{j})) (d_{W}\otimes 1+1\otimes d)\exp(-\sum_{i=1}^{n}\theta^{i}\otimes(\iota_{i}+\sqrt{-1}f_{i}^{j}\iota_{j}))$$
where  $\iota_{a}$ is $\iota_{e_{a}}$.
So we have
$$\delta^{2}=\exp(\sum_{i=1}^{n}\theta^{i}\otimes(\iota_{i}+\sqrt{-1}f_{i}^{j}\iota_{j}))(d_{W}\otimes 1+1\otimes d)
\exp(-\sum_{i=1}^{n}\theta^{i}\otimes(\iota_{i}+\sqrt{-1}f_{i}^{j}\iota_{j}))\cdot$$
$$\exp(\sum_{i=1}^{n}\theta^{i}\otimes(\iota_{i}+\sqrt{-1}f_{i}^{j}\iota_{j}))(d_{W}\otimes 1+1\otimes d)\exp(-\sum_{i=1}^{n}\theta^{i}\otimes(\iota_{i}+\sqrt{-1}f_{i}^{j}\iota_{j}))$$
$$=\exp(\sum_{i=1}^{n}\theta^{i}\otimes(\iota_{i}+\sqrt{-1}f_{i}^{j}\iota_{j}))(d_{W}\otimes 1+1\otimes d)^{2}\exp(-\sum_{i=1}^{n}\theta^{i}\otimes(\iota_{i}+\sqrt{-1}f_{i}^{j}\iota_{j}))$$

$$=0$$
\end{proof}

So we get the BRST differential algebra
$(W(\mathfrak{g})\otimes\Omega^{*}(M),\delta)$.

\begin{lemma}
Fixing the index $i$ and $k$
$$(\theta^{i}\otimes(\iota_{i}+\sqrt{-1}f_{i}^{j}\iota_{j}))(\theta^{k}\otimes(\iota_{k}+\sqrt{-1}f_{k}^{l}\iota_{l}))=(\theta^{k}\otimes(\iota_{k}+\sqrt{-1}f_{k}^{l}\iota_{l}))(\theta^{i}\otimes(\iota_{i}+\sqrt{-1}f_{i}^{j}\iota_{j}))$$
\end{lemma}
\begin{proof}

If $i=k$, we have
$$(\theta^{i}\otimes(\iota_{i}+\sqrt{-1}f_{i}^{j}\iota_{j}))(\theta^{k}\otimes(\iota_{k}+\sqrt{-1}f_{k}^{l}\iota_{l}))=0=(\theta^{k}\otimes(\iota_{k}+\sqrt{-1}f_{k}^{l}\iota_{l}))(\theta^{i}\otimes(\iota_{i}+\sqrt{-1}f_{i}^{j}\iota_{j}))$$
If $i\neq k$, then because
$$(\theta^{i}\otimes\iota_{i})(\theta^{k}\otimes\iota_{k})=-\theta^{i}\wedge\theta^{k}\otimes\iota_{i}\iota_{k}=-\theta^{k}\wedge\theta^{i}\otimes\iota_{k}\iota_{i}=(\theta^{k}\otimes\iota_{k})(\theta^{i}\otimes\iota_{i})$$
$$(\theta^{i}\otimes(\sqrt{-1}f_{i}^{j}\iota_{j}))(\theta^{k}\otimes\iota_{k})=-\theta^{i}\wedge\theta^{k}\otimes(\sqrt{-1}f_{i}^{j}\iota_{j})\iota_{k}=-\theta^{k}\wedge\theta^{i}\otimes\iota_{k}(\sqrt{-1}f_{i}^{j}\iota_{j})=(\theta^{k}\otimes\iota_{k})(\theta^{i}\otimes(\sqrt{-1}f_{i}^{j}\iota_{j}))$$
So we get the result.
\end{proof}

Let
$\psi:W(\mathfrak{g})\otimes\Omega^{*}(M)\rightarrow
W(\mathfrak{g})\otimes\Omega^{*}(M)$ be the map
$$\psi=\prod_{i}(1-\theta^{i}\otimes(\iota_{i}+\sqrt{-1}f_{i}^{j}\iota_{j})).$$
By computation
$$(1-\theta^{1}\otimes(\iota_{1}+\sqrt{-1}f_{1}^{j}\iota_{j}))(1-\theta^{2}\otimes(\iota_{2}+\sqrt{-1}f_{2}^{j}\iota_{j}))\cdots(1-\theta^{n}\otimes(\iota_{n}+\sqrt{-1}f_{n}^{j}\iota_{j}))$$
we have
$$\psi=\exp(-\sum_{i=1}^{n}\theta^{i}\otimes(\iota_{i}+\sqrt{-1}f_{i}^{j}\iota_{j})).$$
In the section 5. we will discuss the map $\psi$.

\section{Weil model}
The exterior derivative operator on
$W(\mathfrak{g})\otimes\Omega^{*}(M)$ is defined by
$$D\doteq d_{W}\otimes 1+1\otimes d,$$
the contraction operator is defined by
$$\widetilde{i}_{X}\doteq i_{X}\otimes 1+1\otimes \iota_{X}$$
and Lie derivative be defined by
$$\widetilde{L}_{X}\doteq L_{X}\otimes 1+1\otimes \mathcal{L}_{X}$$

\begin{lemma}
The operators $\widetilde{i}_{X}, D, \widetilde{L}_{X}$ on
$W(\mathfrak{g})\otimes\Omega^{*}(M)$ satisfy the following
identities:
\begin{description}
\item[(1)]$D^{2}=0$;
\item[(2)]$\widetilde{L}_{X}\cdot D-D\cdot \widetilde{L}_{X}=0$, for any
$X\in\mathfrak{g}$;
\item[(3)]$\widetilde{i}_{X}\widetilde{i}_{Y}+\widetilde{i}_{Y}\widetilde{i}_{X}=0$, for any $X,Y\in\mathfrak{g}$;
\item[(4)]$\widetilde{L}_{X}\widetilde{i}_{Y}-\widetilde{i}_{Y}\widetilde{L}_{X}=\widetilde{i}_{[X,Y]}$, for any
$X,Y\in\mathfrak{g}$;
\item[(5)]$\widetilde{L}_{X}\widetilde{L}_{Y}-\widetilde{L}_{Y}\widetilde{L}_{X}=\widetilde{L}_{[X,Y]}$, for any
$X,Y\in\mathfrak{g}$;
\item[(6)]$\widetilde{L}_{X}=D\cdot\widetilde{i}_{X}+\widetilde{i}_{X}\cdot
D$, for any $X\in\mathfrak{g}$.
\end{description}
\end{lemma}
\begin{proof}
see [4].
\end{proof}

Set $$\widetilde{i}_{X+\sqrt{-1}Y}\doteq i_{X}\otimes 1+1\otimes (\iota_{X}+\sqrt{-1}\iota_{Y})$$
be the contraction operator on $W(\mathfrak{g})\otimes\Omega^{*}(M)$ induced by the contraction of $X+\sqrt{-1}Y$.

Set $$\widetilde{L}_{X+\sqrt{-1}Y}\doteq L_{X}\otimes 1+1\otimes (\mathcal{L}_{X}+\sqrt{-1}\mathcal{L}_{Y})$$
be the Lie derivative on $W(\mathfrak{g})\otimes\Omega^{*}(M)$ about $X+\sqrt{-1}Y$.

\begin{lemma}
$$\widetilde{L}_{X+\sqrt{-1}Y}=D\cdot\widetilde{i}_{X+\sqrt{-1}Y}+\widetilde{i}_{X+\sqrt{-1}Y}\cdot D$$ for any $X,Y\in\mathfrak{g}$.
\end{lemma}
\begin{proof}
\begin{align*}
D\cdot\widetilde{i}_{X+\sqrt{-1}Y}+\widetilde{i}_{X+\sqrt{-1}Y}\cdot D & =(d_{W}\otimes 1+1\otimes d)\cdot\widetilde{i}_{X+\sqrt{-1}Y}+\widetilde{i}_{X+\sqrt{-1}Y}\cdot(d_{W}\otimes 1+1\otimes d)\\
 & = d_{W}i_{X}\otimes1+i_{X}d_{W}\otimes1+1\otimes d(\iota_{X}+\sqrt{-1}\iota_{Y})+1\otimes(\iota_{X}+\sqrt{-1}\iota_{Y})d\\
 & = L_{X}\otimes1+1\otimes(\mathcal{L}_{X}+\sqrt{-1}\mathcal{L}_{Y})\\
 & = \widetilde{L}_{X+\sqrt{-1}Y}
\end{align*}
\end{proof}

\begin{definition}
An element $\eta\in
W(\mathfrak{g})\otimes\Omega^{*}(M)$ is
$\mathbf{basic}$ if it satisfies $\widetilde{i}_{X+\sqrt{-1}Y}\eta=0$,
$\widetilde{L}_{X+\sqrt{-1}Y}\eta=0$ for any $X,Y\in\mathfrak{g}$. Set
$(W(\mathfrak{g})\otimes\Omega^{*}(M))_{bas}$ be
the set of basic elements.
\end{definition}

\begin{lemma}
The operator D preserves
$(W(\mathfrak{g})\otimes\Omega^{*}(M))_{bas}$.
\end{lemma}
\begin{proof}
Set
$\eta\in(W(\mathfrak{g})\otimes\Omega^{*}(M))_{bas}$,
then $\widetilde{i}_{X+\sqrt{-1}Y}\eta=0$ and $\widetilde{L}_{X+\sqrt{-1}Y}\eta=0$ for any
$X,Y\in\mathfrak{g}$. So by Lemma 7., we have
$$(\widetilde{i}_{X+\sqrt{-1}Y}\cdot D)\eta=\widetilde{i}_{X+\sqrt{-1}Y}(D\eta)=\widetilde{L}_{X+\sqrt{-1}Y}\eta-D(\widetilde{i}_{X+\sqrt{-1}Y}\eta)=0$$ for any
$X,Y\in\mathfrak{g}$. \par

And
$$\widetilde{L}_{X+\sqrt{-1}Y}(D\eta)=D(\widetilde{i}_{X+\sqrt{-1}Y}\cdot D)\eta+\widetilde{i}_{X+\sqrt{-1}Y}(D^{2})\eta=0$$
for any $X,Y\in\mathfrak{g}$. \par
Then we get
$$D\eta\in(W(\mathfrak{g})\otimes\Omega^{*}(M))_{bas}.$$
\end{proof}

Now we can construct the cohomology group as following:

By the complex
$((W(\mathfrak{g})\otimes\Omega^{*}(M))_{bas},D)$,
we can define the cohomology group as follow,
$$H^{*}_{G}(M)\doteq\frac{\rm{Ker} D|_{(W(\mathfrak{g})\otimes\Omega^{*}(M))_{bas}}}{\rm{Im}
D|_{(W(\mathfrak{g})\otimes\Omega^{*}(M))_{bas}}}.$$

\begin{definition}
The cohomology group $H^{*}_{G}(M)$ is called the equivariant
cohomology groups of $M$. The equivariant cohomology construct by
this way is called $\mathbf{Weil \ model}$.
\end{definition}

\section{The main results}
In this section we explain the precise relation between the Weil model and the
Cartan model for equivariant cohomology defined earlier.

\begin{theorem}
$\psi$ is an isomorphism of differential algebra, i.e., the diagram
\[
\begin{CD}
W(\mathfrak{g})\otimes\Omega^{*}(M) @>{\psi}>>
W(\mathfrak{g})\otimes\Omega^{*}(M)\\
@V{\delta}VV @VV{D}V\\
W(\mathfrak{g})\otimes\Omega^{*}(M) @>>{\psi}>
W(\mathfrak{g})\otimes\Omega^{*}(M)\\
\end{CD}
\]
commutes.
\end{theorem}
\begin{proof}
By computation in lemma 4., we have $$\delta=\psi\cdot D\cdot\psi^{-1}$$
\end{proof}

\begin{theorem}
We have the following commutative diagram:
\[
\begin{CD}
(W(\mathfrak{g})\otimes\Omega^{*}(M),\delta)
@>{\psi}>>
(W(\mathfrak{g})\otimes\Omega^{*}(M),D)\\
@A{id}AA @AA{id}A\\
(S(\mathfrak{g}^{*})\otimes\Omega^{*}(M))^{\widetilde{G}}
@>>{\psi}>
(W(\mathfrak{g})\otimes\Omega^{*}(M))_{bas}\\
\end{CD}
\]
\end{theorem}
\begin{proof}
For $\forall\alpha\in(S(\mathfrak{g}^{*})\otimes\Omega^{*}(M))^{\widetilde{G}}$,by
$$\prod_{a}(1-\theta^{a}\otimes(\iota_{a}+\sqrt{-1}f_{a}^{b}\iota_{b}))\cdot(i_{k}\otimes1)=(i_{k}\otimes1+1\otimes(\iota_{k}+\sqrt{-1}f_{k}^{j}\iota_{j}))\cdot\prod_{a}(1-\theta^{a}\otimes(\iota_{a}+\sqrt{-1}f_{a}^{b}\iota_{b}))$$
we have
$$(i_{k}\otimes1+1\otimes(\iota_{k}+\sqrt{-1}f_{k}^{j}\iota_{j}))(\psi(\alpha))=0.$$
Because
$$[\delta,i_{k}\otimes1]=L_{k}\otimes1+1\otimes(\mathcal{L}_{k}+\sqrt{-1}f^{j}_{k}\mathcal{L}_{j})$$
and
$$\prod_{a}(1-\theta^{a}\otimes(\iota_{a}+\sqrt{-1}f_{a}^{b}\iota_{b}))\cdot(L_{k}\otimes1+1\otimes(\mathcal{L}_{k}+\sqrt{-1}f^{j}_{k}\mathcal{L}_{j}))$$ $$=(L_{k}\otimes1+1\otimes(\mathcal{L}_{k}+\sqrt{-1}f^{j}_{k}\mathcal{L}_{j}))\cdot\prod_{a}(1-\theta^{a}\otimes(\iota_{a}+\sqrt{-1}f_{a}^{b}\iota_{b}))$$
so we have
$$(L_{k}\otimes1+1\otimes(\mathcal{L}_{k}+\sqrt{-1}f^{j}_{k}\mathcal{L}_{j}))(\psi(\alpha))=0$$
Then we get $\psi(\alpha)\in(W(\mathfrak{g})\otimes\Omega^{*}(M))_{bas}$. So we get the commutative diagram.
\end{proof}

The theorem 2. tell us the relation about BRST model and Cartan model.

\begin{theorem}
\[
\begin{CD}
(S(\mathfrak{g}^{*})\otimes\Omega^{*}(M))^{\widetilde{G}}
@>{\psi}>>
(W(\mathfrak{g})\otimes\Omega^{*}(M))_{bas}\\
\end{CD}
\]
is a isomorphism.
\end{theorem}

\begin{proof}
For $\forall\eta\in(W(\mathfrak{g})\otimes\Omega^{*}(M))_{bas}$,
$\psi^{-1}\eta=\prod_{a}(1+\theta^{a}\otimes(\iota_{a}+\sqrt{-1}f_{a}^{b}\iota_{b}))\eta$.
By $$\prod_{a}(1+\theta^{a}\otimes(\iota_{a}+\sqrt{-1}f_{a}^{b}\iota_{b}))|_{(W(\mathfrak{g})\otimes\Omega^{*}(M))_{bas}}=\prod_{a}(1-\theta^{a}i_{a}\otimes1)|_{(W(\mathfrak{g})\otimes\Omega^{*}(M))_{bas}}$$
and
$${Im}(1-\theta^{a} i_{a}\otimes 1)={Ker}(i_{a}\otimes 1)$$
So $$\psi^{-1}\eta\in (S(\mathfrak{g}^{*})\otimes\Omega^{*}(M))_{bas}.$$
Then $$(\sum_{i=1}^{n}\phi^{i}\otimes(\mathcal{L}_{i}+\sqrt{-1}f_{i}^{j}\mathcal{L}_{j}))\psi^{-1}\eta=0$$
i.e.,$\psi^{-1}\eta\in (S(\mathfrak{g}^{*})\otimes\Omega^{*}(M))^{\widetilde{G}}$.
And by the proof in theorem 2.we get that $\psi$ is a isomorphism.
\end{proof}
The theorem 3. tell us the relation about Cartan model and Weil model.

\end{CJK}

\begin{thebibliography}{10}
\bibitem{AB}
M. Atiyah, R. Bott. The moment map and equivariant cohomology.
Topology 23, 1, 1984.
\bibitem{BGV}
N. Berline, E. Getzler and M. Vergne. {\it Heat Kernels and Dirac
Operators}. Germany: Springer-Verlag, 1992.
\bibitem{Chen}
X. Chen. Localization formulas about two Killing vector Fields. \\ http://arxiv.org/abs/1304.3806v1
\bibitem{GS}
V. W. Guillemin and S. Sternberg. {\it Supersymmetry and equivariant
de Rham theory}. Springer-Verlag, Berlin, 1999.
\bibitem{Kalkman}
J. Kalkman. BRST Model for Equivariant Cohomology and
Representatives for the Equivariant Thom Class. Commun. Math. Phys.,
153: 447-463, 1993.
\bibitem{KN}
S. Kobayashi and K. Nomizu. {\it Foundations of differential
geometry}. Vol. I. John Wiley and Sons, New York(1963).
\bibitem{MQ}
V. Mathai, D. Quillen. Thom classes, superconnections and
equivariant differential forms. Topology 25, 85, 1986.
\bibitem{Szabo}
Richard J. Szabo. {\it Equivariant cohomology and localzition of
Path Integrals}. Springer-Verlag, Berlin, 2000.
\end{thebibliography}
\end{document}